  \documentclass{amsart}
\usepackage{graphicx}
\usepackage{epstopdf}
\usepackage{amscd}
\newtheorem{thm}{Theorem}[section]
\newtheorem{cor}[thm]{Corollary}
\newtheorem{lem}[thm]{Lemma}
\newtheorem{prop}[thm]{Proposition}
\theoremstyle{definition}

\newtheorem{rem}[thm]{Remark}  

\numberwithin{equation}{section}

\newcommand{\R}{\mathbb R}

\newcommand{\To}{\longrightarrow}

\newcommand{\PP}{\mathbb{P}}
\newcommand{\C}{\mathbb{C}}
\newcommand{\Z}{\mathbb{Z}}

\newcommand{\inv}{^{-1}}

\newcommand{\x}{\times}

\newcommand{\B}{{\mathcal B}}

\newcommand{\bd}{{\operatorname{bd}}}

\begin{document}
\title[The totally nonnegative part of $G/P$ is a CW complex]
{The totally nonnegative part of $G/P$ is a CW complex.}
\author{Konstanze Rietsch}%
\author{Lauren Williams}
\address{Department of Mathematics,
            King's College London,
            Strand, London
            WC2R 2LS}%
\address{Department of Mathematics,
            Harvard University,
            1 Oxford Street,
            Cambridge, MA 02138}
\email{konstanze.rietsch@kcl.ac.uk}%
\email{lauren@math.harvard.edu}
\dedicatory{Dedicated to Professor Bertram Kostant on the occasion
of his 80th birthday}
\thanks{
The first author is funded by 
EPSRC grant EP/D071305/1.  The second author is partially 
supported by the NSF}%
%
\begin{abstract}
The totally nonnegative part  of a partial flag variety $G/P$ 
has been shown in \cite{Rie:CelDec,Rie:PhD} to be a union
of semi-algbraic cells. Moreover the closure of a cell was shown
in \cite{Rie:TotPosPoset} to be a union of smaller cells. In this note
we provide glueing maps for each of the cells to prove
that  $(G/P)_{\ge 0}$ is a CW complex.
This generalizes a result of Postnikov, Speyer and the second
author \cite{PSW} for Grassmannians.
 
\end{abstract}
\maketitle

\section{Introduction} 

In a reductive algebraic group over $\C$ split over $\R$ with a fixed
choice of Chevalley generators in the Lie algebra, there is a well-defined
notion of positive, or $\R_{>0}$-valued points due to Lusztig
\cite{Lus:TotPos94}.  In the case of $GL_n$ with the standard choices,
the resulting ``$GL_n(\R_{>0})$'' recovers the classical notion of 
totally positive matrices, that is, matrices such that all minors 
are positive.  For general $G$ the set $G_{>0}$ is therefore 
called the totally positive part of $G$.  The closure $G_{\geq 0}$
of $G_{>0}$ (in the real topology) is called the totally nonnegative
part of $G$.

These notions have a natural extension to flag varieties $G/P$.  That is,
there is a notion of $(G/P)_{>0}$, and of $(G/P)_{\geq 0}$, the 
closure of 
$(G/P)_{>0}$, which Lusztig has described as a 
``remarkable polyhedral subspace" \cite{Lus:TotPos94}.
Lusztig has proved that $(G/P)_{\geq 0}$ is contractible
\cite{Lus:IntroTotPos} and the first author
has shown that it is a union of semi-algebraic cells \cite{Rie:CelDec,Rie:PhD}.
Moreover, 
in \cite{Rie:TotPosPoset} the first author showed that the 
closure of a totally 
nonnegative cell in $G/P$ is a union of totally nonnegative cells
and described the closure relations in terms of the 
Weyl group. The combinatorics of these closure
relations was then studied by the second author in \cite{Williams:poset},
where it was shown that the partially ordered set (poset) of cells of 
$(G/P)_{\geq 0}$ is in fact the poset of cells of a
regular CW complex.  Recall that a CW complex is a union of cells with additional
requirements on how cells are {\it glued}; a  
{\it regular} CW complex is one
where the closure of each cell is homeomorphic to a closed ball
and the closure minus the interior of each cell is homeomorphic to a sphere.
The combinatorial results of \cite{Williams:poset}
prompted the second author to conjecture  
that $(G/P)_{\ge 0}$ is a regular CW complex, which in particular 
would imply that $(G/P)_{\geq 0}$ is homeomorphic to a closed ball.

In \cite{PSW}, Postnikov, Speyer, and the second author proved that
the non-negative part of the Grassmannian is a CW complex, by introducing
an auxiliary toric variety to each parameterization of a cell, and 
constructing a glueing  map from the non-negative part of that toric 
variety to the closure of the corresponding cell.  The construction
of the toric variety and glueing map relied on 
explicit positivity properties
of the parameterizations of the cells, which had been described in terms
of certain graphs in 
\cite{Postnikov}. 

In this paper we generalize the previous result and show
that the non-negative part of any flag variety
$(G/P)_{\ge 0}$ is a CW complex.  As in \cite{PSW}, we again 
construct a toric variety for each parameterization of a cell.  
However, in our proof we use the parameterizations of the cells due to Marsh 
and the first author \cite{MarRie:ansatz}, and use Lusztig's canonical basis \cite{Lus:CanonBasis} 
in order to prove that they have the desired positivity properties. 
Once we have proved that 
$(G/P)_{\ge 0}$ is a CW complex,
the combinatorics from \cite{Williams:poset} 
implies that the closures of the individual cells have Euler
characteristic one. 

The following result is our main theorem.

\begin{thm}\label{th:main}
$(G/P)_{\ge 0}$ is a CW complex.
\end{thm}

In \cite{Williams:poset}, the second author proved the following result.

\begin{thm}\cite{Williams:poset} \label{th:Eulerian}
The poset of cells of $(G/P)_{\ge 0}$ is the poset of cells of some
regular CW complex; therefore the poset
of cells of $(G/P)_{\ge 0}$ is Eulerian.
\end{thm}

In other words, the alternating sum of cells in the closure of a cell 
of $(G/P)_{\ge 0}$ is $1$.  
This result combined with Theorem \ref{th:main} implies the following.

\begin{cor}
The Euler characteristic of the closure of a cell of 
$(G/P)_{\ge 0}$ is $1$.  
\end{cor}

The structure of this paper is as follows.  In Section \ref{s:prelim}, we review
basic results on algebraic groups and flag varieties.  In Sections \ref{s:TotPosGB}
and \ref{s:Toric} we introduce the notion of total positivity for real reductive groups
and flag varieties, and toric varieties, respectively.  In Section \ref{s:ConstructToric}
we construct a toric variety associated to a parameterization of a cell.  In Section
\ref{s:proof} we prove a key proposition,
and 
in Section \ref{s:Partial}, we prove the main result.


\section{Preliminaries}\label{s:prelim}
\subsection{}
We recall some basic notation and results from algebraic groups,
see e.g. \cite{Springer:AlgGroupBook}. 
Let $G$ be a simply connected  semisimple linear algebraic group over 
$\C$ split over $\R$. We identify $G$ and any related spaces with their
$\R$-valued points in  their real topology (as real manifolds or subsets thereof).   
We write $\R^*$ for $\R\setminus\{0\}$. 

Let $T$ be a split torus and $B^+$ and $B^-$ opposite Borel 
subgroups containing $T$. We denote the character and the cocharacter
groups of $T$ by $X^*(T)$ and $X_*(T)$, respectively. Let $<\ ,\ >$
denote the dual pairing between $X^*(T)$ and $X_*(T)$. The unipotent 
radicals of $B^+$ and 
$B^-$ are denoted $U^+$ and $U^-$, respectively. 
Let $\{\alpha_i\ | \ i\in I\}\subset X^*(T)$ be the set of simple roots
associated to $B^+$ and $\{\alpha_i^\vee\ |\ i\in I\}\subset X_*(T)$ the 
corresponding coroots. Then we have the simple root subgroups 
$U^+_{\alpha_i}\subseteq U^+$ and $U^-_{\alpha_i}\subseteq U^-$.
Furthermore assume we are given homomorphisms
\begin{equation*}
\phi_i: SL_2(\mathbb \R)\to G, \qquad i\in I,
\end{equation*}  
such that 
\begin{equation*}
\phi_i\left(\begin{pmatrix}t &0\\ 0& t\inv 
\end{pmatrix}\right)=\alpha_i^\vee(t), \qquad t\in \R^*,
\end{equation*}
and such that 
\begin{equation*}
\phi_i\left(\begin{pmatrix}1 & m\\ 0&1\end{pmatrix}\right):=x_i(m), 
\quad 
\phi_i\left(\begin{pmatrix}1 & 0\\ m&1\end{pmatrix}\right):=y_i(m), 
\end{equation*}
define isomorphisms $x_i:\R\to U^+_{\alpha_i}$ and $y_i:
\R\to U^-_{\alpha_i}$. Following \cite{Lus:TotPos94}, the datum
$(T,B^+,B^-, x_i,y_i, i\in I)$ is called a {\it pinning} for $G$. 

\subsection{}\label{s:automorphism} If $G$ is not simply laced, then one can
construct a simply laced group $\dot G$ and an automorphism
$\tau$ of $\dot G$ defined over $\R$, such that there is
an isomorphism, also defined over $\R$, between $G$ and
 the fixed point subset $\dot G^\tau$ of $\dot G$. Moreover
 the groups $G$ and $\dot G$ have compatible pinnings.  Explicitly
 we have the following.

Let $\dot G$ be simply connected and simply laced.
We apply the same notations as above for $G$, but with an added  dot,
to our simply laced group $\dot G$. So we have a pinning 
$(\dot T,\dot B^+,\dot B^-, \dot x_i,\dot y_i, i\in \dot I)$ of $\dot G$,
and $\dot I$ may be identified with the vertex set of the Dynkin
diagram of $\dot G$. 

Let $\sigma$ be a permutation of $\dot I$ preserving
connected components of the Dynkin diagram, such that,
if $j$ and $j'$ lie in  the same orbit under $\sigma$
then they are {\it not} connected by an edge.   
Then $\sigma$ determines an automorphism $\tau$ of $\dot G$
such that 
\begin{enumerate}
\item
$\tau(\dot T)=\dot T$,
\item
$\tau(x_i(m))=x_{\sigma(i)}(m)$ and $\tau(y_i(m))=y_{\sigma(i)}(m)$ for all 
$i\in \dot I$ and $m\in \R$.
\end{enumerate}
In particular $\tau$ also preserves $\dot B^+,\dot B^-$. Let $\bar I$ denote 
the set of $\sigma$-orbits in $\dot I$, and for $\bar i\in \bar I$, let 
\begin{eqnarray*}
x_{\bar i}(m)&:=&\prod_{i\in \bar i}\ x_i(m),\\
y_{\bar i}(m)&:=&\prod_{i\in \bar i}\ y_i(m). 
\end{eqnarray*}
The fixed
point group $\dot G^{\tau}$ is a simply connected algebraic group with  pinning
$({\dot T}^\tau,\dot B^{+ \tau},\dot B^{-\, \tau},  x_{\bar i}, y_{\bar i}, \bar i\in \bar I)$.
There exists, and we choose, $\dot G$ and $\tau$ such that  $\dot G^{\tau}$
is isomorphic to our group $G$ via an isomorphism compatible with the pinnings.

\subsection{}
Let $W=N_G(T)/T$ be the Weyl group of $G$. For $i\in I$ the elements
\begin{equation*}
\dot s_i=x_i(-1)y_i(1)x_i(-1)
\end{equation*} 
represent the simple reflections $s_i\in W$. If $w=s_{i_1}\dotsc s_{i_m}$ 
is a reduced expression for $w$ then we write $\ell(w)=m$ for the length
of $w$. We note also that the representative
\begin{equation*}
\dot w=\dot s_{i_1}\dotsc \dot s_{i_m}
\end{equation*}   
of $w$ is well-defined, independent of the reduced expression. 
Inside $W$ there is a longest element which is denoted by $w_0$.

\subsection{}
Let $J$ be a subset of $I$. The parabolic subgroup 
$W_J\subseteq W$ is the subgroup generated by all of the
$s_j$ with $j\in J$. Let $w_J$ denote the longest element in 
$W_J$.  We also consider the set $W^J$
of minimal-length coset representatives for $W/W_J$, and the
set $W^J_{max}=W^Jw_J$ of maximal-length coset representatives.

The parabolic subgroup $W_J$ of $W$
corresponds to a parabolic subgroup $P_J$ 
in $G$ containing $B^+$. Namely, 
$P_J$ is the subgroup of $G$ generated by $B^+$ and the
elements $\dot w$ for $w\in W_J$.
Let  $\mathcal P^J$ be the set of parabolic subgroups $P$  
conjugate to $P_J$. This is a homogeneous space for the conjugation 
action of $G$ and can be identified with the 
partial flag variety $G/P_J$ via
 $$
G/P_J\overset\sim\To \mathcal P^J\ :\ gP_J\mapsto gP_J g\inv.
 $$
In the case $J=\emptyset$ we are identifying the full flag variety $G/B^+$
with the variety $\mathcal B$ of Borel subgroups in $G$. We have the
usual projection from the full flag variety to any 
partial flag variety which takes the form 
$\pi=\pi^J:\mathcal B\to\mathcal P^{J}$, where $\pi(B)$ is the 
unique parabolic subgroup of type $J$ containing $B$. 

The conjugate of a parabolic subgroup $P$ by 
an element $g\in G$ will be denoted by
$g\cdot P:=gPg\inv$.  

\subsection{}
Recall the Bruhat decomposition for the full flag variety,
$$
\mathcal B=\bigsqcup_{w\in W} B^+\dot w\cdot B^+,
$$
and the Bruhat order $\le$ on $W$. 
The Bruhat cell $B^+\dot w\cdot B^+$ is isomorphic
to $\mathbb \R^{\ell(w)}$. And the Bruhat order has the property
$$
v\le w\ \  \iff \ \  
B^+\dot v\cdot B^+\subseteq\overline {B^+\dot w\cdot B^+},
$$
for $v,w\in W$.

It is a well-known consequence of Bruhat decomposition 
that $\mathcal B\x\mathcal B$ is the union of the 
$G$-orbits $\mathcal O(w)=G\cdot (B^+,\dot w\cdot B^+)$,
with $G$ acting
diagonally. Therefore to any pair $(B_1,B_2)$ of Borel 
subgroups one can associate a unique $w\in W$ such that
 \begin{equation*}
(B_1,B_2)=(g\cdot B^+,g\dot w \cdot B^+)
\end{equation*} 
for some $g\in G$. We write  
$$
B_1\overset w\to B_2
$$ 
in this case and call $w$ the relative position of $B_1$ and $B_2$.

\subsection{}
Finally, let us consider the two opposite Bruhat 
decompositions 
\begin{equation*}
\mathcal B=\bigsqcup_{w\in W} B^+\dot w\cdot B^+=\bigsqcup_{v\in W}
B^-\dot v\cdot B^+.
\end{equation*}
Note that $B^-\dot v\cdot B^+\cong\mathbb R^{\ell(w_0)-\ell(v)}$. The 
closure relations for these opposite Bruhat cells are 
given by $B^-\dot v'\cdot B^+\subset 
\overline{B^-\dot v\cdot B^+}$ if and only if $v\le v'$. 
We define
\begin{equation*}
\mathcal R_{v,w}:=B^+\dot w\cdot B^+\cap B^-\dot v\cdot B^+,
\end{equation*}
the intersection of opposed Bruhat cells. This intersection is empty
unless $v\le w$, in which case it is smooth of dimension 
$\ell(w)-\ell(v)$, see \cite{KaLus:Hecke2,Lus:PartFlag}.

\section{Total Positivity for $G$ and $\mathcal B$}\label{s:TotPosGB}

Real projective space $\PP^n$ has a natural open subset: the set of 
lines spanned by vectors with all coordinates positive.  This subset is 
called the totally positive part of $\PP^n$, and its closure, the set
of lines spanned by vectors with all coordinates non-negative, is called
the totally non-negative part of $\PP^n$.  These subsets can be defined
more generally \cite{Lus:TotPos94}
for any split semisimple real algebraic group and any 
partial flag manifold  of such a group.

\subsection{Total positivity in $G$}\label{s:totpos}The totally nonnegative part $G_{\ge 0}$ of $G$ is defined
by Lusztig \cite{Lus:TotPos94} to be the semigroup
inside $G$ generated by the sets
\begin{align*}
&\{x_i(t)\ |\ t\in \R_{>0}, i\in I\},\\
&\{y_i(t)\ |\ t\in \R_{>0}, i\in I\}, \text{ and}\\
&T_{>0}:=\{ t\in T\ |\ \text{$\chi(t)>0$ all $\chi\in X^*(T)$}\}.
\end{align*}
When $G=SL_n(\R)$ then by a theorem of A.~Whitney 
this definition agrees with the classical notion 
of totally nonnegative matrices inside $SL_n(\R)$, 
that is those matrices all of whose minors are 
nonnegative.

We recall some basic facts about 
total positivity for $G$ from \cite{Lus:TotPos94}.  
Let 
$U^+_{\ge 0}:=G_{\ge 0}\cap U^+$  and $U^-_{\ge 0}:=G_{\ge 0}\cap U^-$.
For $w\in W$ and $s_{i_1}\dotsc s_{i_m}=w$ a reduced expression define
\begin{eqnarray*}
U^+(w)&:=&\{x_{i_1}(t_1)x_{i_2}(t_2)\dotsc x_{i_m}(t_m)\ |\ t_i\in\R_{>0}\},\\
U^-(w)&:=&\{y_{i_1}(t_1)y_{i_2}(t_2)\dotsc y_{i_m}(t_m)\ |\ t_i\in\R_{>0}\}.
\end{eqnarray*} 
These sets are independent of the chosen 
reduced expression and give
\begin{eqnarray*}
U^+(w)&=&U^+_{\ge 0}\cap B^-\dot w B^-,\\
U^-(w)&=& U^-_{\ge 0}\cap B^+\dot w B^+.
\end{eqnarray*}
In particular $U^+_{\ge 0}=\bigsqcup_{w\in W} U^+(w)$ and $U^-_{\ge 0}=
\bigsqcup_{w\in W}U^-(w)$. Moreover $U^+(w)$ and $U^-(w)$ are isomorphic to 
$\R_{>0}^{\ell(w)}$ using the $t_i$ as coordinates. 
The totally positive parts for $U^+$ and $U^-$ 
are defined by 
\begin{equation*}
U^+_{>0}:=U^+(w_0),\qquad U^-_{>0}:=U^-(w_0).
\end{equation*}

\subsection{Total positivity in $\mathcal B$}\label{s:totposflag}
The totally positive and totally nonnegative parts of the 
flag variety $\mathcal B$ are defined by 
\begin{eqnarray*}
\mathcal B_{>0}&:=&\{y\cdot B^+\ |\ y\in U^-_{>0}\},\\
\mathcal B_{\ge 0}&:=&\overline{\mathcal B_{>0}}.
\end{eqnarray*}

The set $\mathcal B_{\ge 0}$ has a 
decomposition into strata,  
\begin{equation*}
\mathcal R^{>0}_{v,w}:=\mathcal R_{v,w}\cap \mathcal B_{\ge 0},
\end{equation*}
where $v\le w$. 
These strata were defined and conjectured
to be semi-algebraic cells by Lusztig \cite{Lus:TotPos94}, 
a result which was later  proved  in \cite{Rie:CelDec}.
The conjecture was proved again in a different way in \cite{MarRie:ansatz},
this  time with explicit parametrizations of the cells
given. We recall these parametrizations 
now. 

Let $v\le w$ and let $\mathbf w=(i_1,\dotsc, i_m)$ encode a
reduced expression $s_{i_1}\dotsc s_{i_m}$ for $w$. Then there
exists a unique subexpression $s_{i_{j_1}}\dotsc s_{i_{j_k}}$ 
for $v$ in $\mathbf w$ with the 
property that, for $l=1,\dotsc,k$,
\begin{equation*}
s_{i_{j_1}}\dotsc s_{i_{j_{l}}}s_{i_r}>
s_{i_{j_1}}\dotsc s_{i_{j_{l}}}  \quad \text{whenever $j_l< r\le j_{l+1}$,}
\end{equation*}
where $j_{k+1}:=m$. This is loosely speaking the rightmost reduced subexpression for $v$ in 
$\mathbf w$. It is called the  `positive subexpression' in \cite{MarRie:ansatz}, and 
we use the notation
\begin{eqnarray*}
\mathbf v_+&:=&\{j_1,\dotsc, j_k\},\\
\mathbf v_+^c&:=&\{1,\dotsc, m\}\setminus \{j_1,j_2,\dotsc, j_k\},
\end{eqnarray*}
when referring to this special subexpression for $v$ in $\mathbf w$.

Now we can define a map 
\begin{eqnarray*}
\phi_{\mathbf v_+,\mathbf w}:(\C^*)^{\mathbf v^c_+} &\to& \mathcal R_{v,w},\\
        ( t_r)_{r\in\mathbf v^c_+}&\mapsto &g_1\dotsc g_m\cdot B^+,
\end{eqnarray*}
where
\begin{equation*}
g_r=\begin{cases}
\dot s_{i_r}, &\text{ if $r\in\mathbf v_+$,}\\
y_{i_r}(t_r) & \text{ if $r\in \mathbf v^c_+$.}
\end{cases}  
\end{equation*}
It is shown in \cite{MarRie:ansatz} that the map $\phi_{\mathbf v_+,\mathbf w}$ 
is an embedding with image
the open `Deodhar stratum' of $\mathcal R_{v,w}$ associated to $\mathbf w$,
see \cite{Deo:Decomp}. 

\begin{thm}\cite[Theorem~11.3]{MarRie:ansatz}\label{t:param} 
The restriction of $\phi_{\mathbf v_+,\mathbf w}$ to $(\R_{>0})^{\mathbf v_+^c}$ 
defines an isomorphism of semi-algebraic sets,
$$
\phi^{>0}_{\mathbf v_+,\mathbf w}: (\R_{>0})^{\mathbf v_+^c} \to \mathcal R_{v,w}^{>0}.
$$
\end{thm}


\subsection{Changes of coordinates under braid relations} \label{s:BraidRels} 

\bigskip

In the simply laced case 
there is a simple change of coordinates \cite{Lus:TotPos94, Rie:MSgen}
which describes how 
two parameterizations of the same cell are related when considering
two reduced expressions which differ by a commuting relation or 
a braid relation.

If $s_i s_j = s_j s_i$ then $y_i(a) y_j(b)=y_j(b) y_i(a)$ 
   and $y_i(a) \dot s_j = \dot s_j y_i(a)$.

If $s_i s_j s_i = s_j s_i s_j$ then
\begin{enumerate}
\item 
$y_i(a) y_j(b) y_i(c) = y_j(\frac{bc}{a+c}) y_i(a+c) y_j(\frac{ab}{a+c})$.\\
\item $y_i(a) \dot s_j y_i(b) = y_j(\frac{b}{a}) y_i(a) \dot s_j$\\
\item 
    $\dot s_j \dot s_i y_j(a)=
    y_i(a) \dot s_j \dot s_i$.

\end{enumerate}

The changes of coordinates have also 
been computed for more general braid relations
and have been observed to be subtraction-free \cite{BZ,Rie:MSgen}.



\subsection{Closure relations}
We have the following closure relations.

\begin{thm}{\cite{Rie:TotPosPoset}}\label{t:B} Let $v,w\in W$. Then 
$$
\overline{\mathcal R_{v,w}^{>0}}=\bigsqcup_{ 
v\le v'\le w'\le w} 
\mathcal R^{>0}_{v',w'}.
$$ 
\end{thm}

\subsection{Total positivity and canonical bases, for simply laced $G$}\label{s:CanonBasis}


%

\bigskip

We now assume that $G$ is simply laced.  Let $\bf U$ be the enveloping
algebra of the Lie algebra of $G$; this can be defined by generators
$e_i, h_i, f_i$ ($i\in I$) and the Serre relations.  For any 
dominant character
$\lambda$ 
there is a finite-dimensional simple $\bf U$-module $V(\lambda)$
with a non-zero vector $\eta$ such that $e_i \eta = 0$
and $h_i \eta = <\lambda,\alpha_i^\vee> \eta$ for all $i \in I$.  Moreover, the 
pair $(V(\lambda), \eta)$ is determined up to unique isomorphism.

There is a unique $G$-module structure on $V(\lambda)$ such that
for any $i\in I, a\in \R$ we have that $x_i(a)$ acts by 
\begin{equation*}
 \exp(a e_i):V(\lambda) \to V(\lambda), 
\end{equation*}
and $y_i(a)$ acts by
\begin{equation*}
\exp(a f_i): V(\lambda) \to V(\lambda).
\end{equation*}
Then 
$x_i(a)\eta = \eta$ for all $i\in I$, $a\in \R$, and 
$t \eta = \lambda(t) \eta$ for all $t \in T$.  Let $\B(\lambda)$
be the canonical basis of $V(\lambda)$ that contains $\eta$.
See \cite{Lus:CanonBasis} for details on the canonical basis.
In relation to $G_{\geq 0}$, the basis 
$\B(\lambda)$ has the following positivity property 
\cite[Prop. 3.2]{Lus:TotPos94}.

\begin{thm} \cite{Lus:TotPos94} 
Let $g\in G_{\geq 0}$.  Then the 
matrix entries of $g: V \to V$ with respect to $\B(\lambda)$ are non-negative
real numbers.
\end{thm}

\begin{rem}\label{rem:positive1}
We remark that it is easy to see that a little more is true.  
The 
matrix entries of any $x_i(a)$ or 
$y_i(a): V \to V$ with respect to $\B(\lambda)$ 
are given by positive polynomials in $a$.
\end{rem}

The following lemma comes from \cite{Lus:PartFlag}.  

\begin{lem}\label{lem:positive2} \cite[1.7(a)]{Lus:PartFlag}.
For any $w\in W$, the vector $\dot w \eta$ is the unique element of 
$\B(\lambda)$ which lies in the extremal weight space 
$V(\lambda)^{w(\lambda)}$.  In particular, $\dot w \eta \in \B(\lambda)$.
\end{lem}

\section{Positivity for toric varieties}\label{s:Toric}

In this section we define projective toric varieties and recall 
some basic results.

We may define a (generalized) projective toric variety 
as follows \cite{Cox, Sottile}.  
Let $S=\{\mathbf{m}_i \ \vert \ i=1, \dots, \ell\}$ be any finite subset 
of $\Z^n$, where $\Z^n$ can be thought of as the character group
of the torus $(\C^*)^n$. 
Here 
$\mathbf{m}_i=(m_{i1}, m_{i2},\dots ,m_{in})$.
Then consider
the map $\phi: (\C^*)^n \to \PP^{\ell-1}$ such that
$\mathbf{x}=(x_1, \dots , x_n) \mapsto [\mathbf{x^{m_1}}, \dots , \mathbf{x^{m_\ell}}]$.
Here $\mathbf{x^{m_i}}$ denotes $x_1^{m_{i1}} x_2^{m_{i2}} \dots x_n^{m_{in}}$.
We then define the toric variety $X_S$
to be the Zariski closure of the image of this map.
The {\it real part} $X_S(\R)$ of $X_S$ is defined to be the
intersection of $X_S$ with $\R\PP^{\ell-1}$; the
{\it positive part} $X_S^{>0}$ is defined to be the image of
$(\R_{>0})^n$ under $\phi$; and the {\it non-negative part}
$X_S^{\geq 0}$ is defined to be the closure (in $X_S(\R)$) of
$X_S^{>0}$.

Note that $X_S$ is not necessarily a toric variety in the sense of \cite{Fulton}, as 
it may not be normal;
however, its normalization is a toric variety in this sense.  See \cite{Cox} for more details.

Observe that if $S$ is the set of lattice points in the standard simplex, then 
 $X_S^{>0}$ and 
$X_S^{\geq 0}$ are the totally positive and totally non-negative 
parts of real projective space.

Let $P$ be the convex hull of $S$.  
The restriction of the moment map is a homeomorphism from $X_S^{\geq 0}$ to $P$
(see \cite[Section 4.2, page 81]{Fulton}
and \cite[Theorem 8.4]{Sottile}).
In particular, 
$X_S^{\geq 0}$ is homeomorphic to a closed ball.

The following lemma, proved in \cite{PSW}, will be an important tool in the proof of our
main result.

\begin{lem}\label{important} \cite{PSW}
Suppose we have a map $\Phi: (\R_{>0})^n \to \PP^{N-1}$ given by
\begin{equation*}
(t_1, \dots , t_n) \mapsto [h_1(t_1,\dots,t_n), \dots , h_N(t_1,\dots,t_n)],
\end{equation*}
where the $h_i$'s are Laurent polynomials with positive coefficients.  Let $S$ be the
set of all exponent vectors in $\Z^n$ which occur among the (Laurent) monomials
of the $h_i$'s, and let $P$ be the convex hull of the points of $S$.
Then the map $\Phi$ factors through the totally positive part
$X_S^{>0}$, giving a map 
$\tau_{>0}: X_S^{>0} \to \PP^{N-1}$.  Moreover $\tau_{>0}$ extends continuously to the 
closure to give a well-defined map  
$\tau_{\ge 0}:X_S^{\ge 0} \to \overline{\tau_{>0}(X_{S}^{>0})}$.
\end{lem}

\begin{proof}
Let $S = \{\mathbf{m_1},\dots,\mathbf{m_{\ell}}\}$.
Clearly the map $\Phi$ factors
as the composite map $t=(t_1,\dots,t_n) \mapsto
 [\mathbf{t^{m_1}}, \dots , \mathbf{t^{m_\ell}}] \mapsto [h_1(t_1,\dots,t_n),\dots,
       h_N(t_1,\dots,t_n)]$,
and the image of $(\R_{>0})^n$ under the first map is precisely 
$X_S^{>0}$.
The second map, which we will call $\tau_{>0}$, 
takes a point $[x_1,\dots, x_{\ell}]$ of $X_S^{>0}$ to
   $[g_1(x_1,\dots,x_{\ell}), \dots, g_N(x_1,\dots,x_{\ell})]$,
where the $g_i$'s are homogeneous polynomials of degree $1$ with positive coefficients.
By construction, each $x_i$ occurs in at least one of the $g_i$'s.

Since $X_S^{\geq 0}$ is the closure inside $X_S$ of $X_S^{>0}$,
any point $[x_1,\dots,x_{\ell}]$ of $X_S^{\geq 0}$ has all $x_i$'s non-negative;
furthermore, not all of the $x_i$'s are equal to $0$.  And now since the $g_i$'s
have positive coefficients and they involve {\it all} of the $x_i$'s, the image of
any point $[x_1,\dots,x_{\ell}]$ of $X_S^{\geq 0}$ under $\tau_{>0}$ is well-defined.
Therefore $\tau_{>0}$ extends continuously to the closure to give a well-defined map
$\tau_{\ge 0}:X_S^{\ge 0} \to \overline{\tau_{>0}(X_{S}^{>0})}$.

\end{proof}

\bigskip
\section{Construction of a toric variety associated to a parametrization of  a cell}\label{s:ConstructToric}

We begin by stating a key proposition. 

\begin{prop}\label{p:main} Given $G$ we can construct a positivity preserving embedding 
$i:G/B\to \mathbb P^N$, for some $N$ with the following property.
For any totally nonnegative cell $\mathcal R_{v,w}^{>0}$ and parameterization
$\phi^{>0}_{\mathbf v_+, \mathbf w}$ as in Theorem~\ref{t:param}, the composition 
$$
i\circ\phi^{>0}_{\mathbf v_+, \mathbf w}:(\R_{>0})^{\mathbf v^c_+}\overset{\sim}\To\mathcal R_{v,w}^{>0}\hookrightarrow
\mathbb P^N
$$
takes the form 
$$
i \circ\phi^{>0}_{\mathbf v_+, \mathbf w} : \mathbf t=  ( t_r)_{r\in\mathbf v^c_+}\mapsto [p_1(\mathbf t),\dotsc,p_{N+1}(\mathbf t)],
$$
where the $p_j$'s are polynomials with nonnegative coefficients. 
\end{prop}

This proposition will be proved in the next section. 
Now assuming Proposition \ref{p:main} is true, we can prove the following.

\begin{cor} \label{c:glue}
There is a map $\tau_{>0}:X_{\mathbf v_+,\mathbf w}^{>0}\to \PP^N$ which extends continuously to 
the closure to give a well-defined map
$$
\tau_{\ge 0}:X_{\mathbf v_+,\mathbf w}^{\ge 0}\to \overline{\tau_{>0}(X_{\mathbf v_+,\mathbf w}^{>0})}.
$$
Moreover we have
$$\overline{\tau_{>0}(X_{\mathbf v_+,\mathbf w}^{>0})}=i(\overline{\mathcal R_{v,w}^{>0}})
\overset\sim\To \overline{\mathcal R_{v,w}^{>0}}.
$$
The resulting map $X_{\mathbf v_+,\mathbf w}^{\ge 0}\to  \overline{\mathcal R_{v,w}^{>0}}$
is surjective, an isomorphism on the strictly positive parts, and takes the boundary of $X_{\mathbf v_+,\mathbf w}^{\ge 0}$ to the boundary of  $\overline{\mathcal R_{v,w}^{>0}}$. 
\end{cor}

\begin{proof}
Let
$S_{\mathbf v_+, \mathbf w}$ be the set of all exponent vectors in $\Z^{n+1}$
which occur among the monomials of the $p_i$'s, and let 
$P_{\mathbf v_+,\mathbf w}$ be the convex hull of the points of $S_{\mathbf v_+, \mathbf w}$.
Let 
$X_{\mathbf v_+,\mathbf w}$  be the toric variety associated with 
$P_{\mathbf v_+,\mathbf w}$.  By Lemma \ref{important}, 
the map 
$i \circ\phi^{>0}_{\mathbf v_+, \mathbf w}$ factors through 
$X_{\mathbf v_+,\mathbf w}^{> 0}$, 
\begin{equation*}
(\R_{>0})^{\mathbf v^c_+}\overset{\sim}\To X_{\mathbf v_+,\mathbf w}^{> 0}\To\mathcal R_{v,w}^{>0}\hookrightarrow
\mathbb P^N,
\end{equation*}
and  
we get a 
map 
$\tau_{>0}$ from
$X_{\mathbf v_+,\mathbf w}^{> 0}$ to $\PP^N$.  Moreover, this map extends 
continuously to a map $\tau_{\geq 0}$ from 
$X_{\mathbf v_+,\mathbf w}^{\geq 0}$ to  
$\overline{\tau_{>0}(X_{\mathbf v_+,\mathbf w}^{>0})}$.

Since the map 
${i \circ\phi^{>0}_{\mathbf v_+, \mathbf w}}$ is a homeomorphism onto its image
and it factors through 
$X_{\mathbf v_+,\mathbf w}^{> 0}$ as above,  
the map $\tau_{\geq 0}$ restricts to a homeomorphism from 
$X_{\mathbf v_+,\mathbf w}^{> 0}$ to 
$i({\mathcal R_{v,w}^{>0}})$.
We now claim that that $\tau_{\geq 0}$ 
takes the boundary of $X_{\mathbf v_+,\mathbf w}^{\ge 0}$ 
to the boundary of  $\overline{i(\mathcal R_{v,w}^{>0})}$. 

To prove the claim, suppose that there is a point $x \in \bd(X_{\mathbf v_+,\mathbf w}^{\ge 0})$ such that
$\tau_{\geq 0}(x)=y$ is in the interior of $\overline{i(\mathcal R_{v,w}^{>0})}$.  Since $x$ is in the boundary of 
$X_{\mathbf v_+,\mathbf w}^{\ge 0}$ we can find
a sequence of points $\{x_i\}$ in $X_{\mathbf v_+,\mathbf w}^{> 0}$ which converge to $x$.  Let 
$y_i=\tau_{\geq 0}(x_i)$.
Since $\tau_{\geq 0}$ is a homeomorphism on the interior, each $y_i$ is in $i({\mathcal R_{v,w}^{>0}})$, 
and since
$\tau_{\geq 0}$ is continuous, the sequence of points $y_i$ converges to $y$.  
Let $g:i({\mathcal R_{v,w}^{>0}}) \overset \sim\to X_{\mathbf v_+,\mathbf w}^{> 0}$ denote
the inverse of the restriction of
$\tau_{\geq 0}$ to the interiors.  We now have that 
$g$ maps the
convergent sequence $\{y_i\}$ in $i({\mathcal R_{v,w}^{>0}})$ to the
divergent sequence $\{x_i\}$ in $X_{\mathbf v_+,\mathbf w}^{> 0}$, which 
contradicts the continuity
of $g$.

\end{proof}

Since $X_{\mathbf v_+,\mathbf w}^{\ge 0}$ is homeomorphic to a closed ball  the above corollary
provides us with a glueing map for $\mathcal R_{v,w}^{>0}$.

\bigskip

\section{proof of Proposition \ref{p:main}}\label{s:proof}

In this section we will prove Proposition \ref{p:main}.

\subsection{Simply laced case}\label{s:simplylaced}
We suppose $G$ is simply laced. Consider the 
representation 
$V=V(\rho)$ of $G$ with
a fixed highest weight vector $\eta$ and 
corresponding canonical basis $\mathcal B(\rho)$. 
Let $i:\mathcal B\to \mathbb P(V)$ denote the embedding
which takes $g\cdot B_+\in \mathcal B$ to the line 
$\left <g\cdot \eta\right >$. 
This is the unique $g \cdot B_+$-stable line in $V$.
We specify points in the projective space $\mathbb P(V)$
using homogeneous coordinates corresponding to
$\mathcal B(\rho)$.

Now let $\mathbf w_0=(i_1,\dotsc, i_N)$ be a fixed reduced expression of $w_0$. 

\begin{lem}\label{l:w0} Let $v\in W$, and let $\mathbf v_+$ be the positive subexpression for 
$v$ in $\mathbf w_0$. The composition 
$$
i\circ \phi_{\mathbf v_+,\mathbf w_0}^{>0}: (\R_{>0})^{\mathbf v_+^c}\to \mathcal R_{v,w_0}^{>0} \to
\mathbb P(V)
$$
is given by polynomials with positive coefficients. 
\end{lem}

\begin{proof}
Consider first a reduced expression of $w_0$ which ends in $v$, i.e.\
consider a reduced expression $s_{i_1} \dots s_{i_k} s_{i_{k+1}} \dots
s_{i_m}$ of $w_0$, where $s_{i_{k+1}} \dots s_{i_m}$
is a reduced expression of $v$  (which is clearly positive distinguished).
Let $\dot v = \dot s_{i_{k+1}} \dots \dot s_{i_m}$.  Then  
$i\circ \phi_{\mathbf v_+,\mathbf w_0}^{>0}$
maps $(a_1, \dots , a_k)$ to 
$\left <y_{i_1}(a_1) \dots y_{i_k}(a_k) \dot v \cdot \eta\right >$. 
By Lemma \ref{lem:positive2}, $\dot v \cdot \eta$ is a canonical basis vector,
and so by Remark \ref{rem:positive1}, the coefficients which express
$\left <y_{i_1}(a_1) \dots y_{i_k}(a_k) \dot v \cdot \eta\right >$ 
in terms of the canonical basis are positive polynomials in the $a_i$'s.

Now let the reduced expression for $w_0$ be arbitrary.  
This reduced expression can be obtained from our previous
reduced expression $s_{i_1} \dots s_{i_k} s_{i_{k+1}} \dots
s_{i_m}$ using braid relations, 
so the results in Section~\ref{s:BraidRels} imply that 
the parameterization in question will change by a sequence of 
substitutions which involve only positive subtraction-free 
rational expressions.  It follows that the 
image of the map $i\circ \phi_{\mathbf v_+,\mathbf w_0}^{>0}$
is given by rational expressions with positive coefficients.
Since we are in projective space, we can clear denominators to 
obtain polynomials with positive coefficients.

\end{proof}

This proves Proposition~\ref{p:main} for simply laced $G$ 
in the case where $w=w_0$. 
We now turn to the case of arbitrary $w$.

\begin{prop}\label{p:mainsimplylaced}
Proposition~\ref{p:main} holds when $G$ is simply laced.
\end{prop}
\begin{proof}
Choose $v,w\in W$ with $v \leq w$.
Choose a reduced expression $\mathbf w_0 = (i_1, \dots , i_N)$
for $w_0$ such that $(i_1, \dots , i_r)$ is a reduced expression
for $w_0 w\inv$.  Then as in the proof of Lemma 4.3 in 
\cite{Rie:TotPosPoset}, we have 
\begin{equation*}
\mathcal R^{>0}_{v,w_0} = 
\{y_{i_1}(t_1) \dots y_{i_r}(t_r) g_{r+1} \dots g_N \cdot B^+ \ \vert \
g_{r+1} \dots g_N \cdot B^+ \in 
\mathcal R^{>0}_{v,w} \}.
\end{equation*}
  Here each $g_{r+j}$ is either 
$\dot s_{i_{r+j}}$ or $y_{i_{r+j}}(t_{r+j})$ and the $g$'s give 
a parameterization of 
$\mathcal R^{>0}_{v,w} $.   
It's clear that our parameterization
of 
$\mathcal R^{>0}_{v,w}$ is obtained from the above parameterization of 
$\mathcal R^{>0}_{v,w_0}$ by setting $t_1, \dots t_r$ to $0$.
Since setting certain variables to zero in a positive polynomial 
results in another positive polynomial, 
Proposition~\ref{p:main} in the simply laced case now follows from Lemma \ref{l:w0}.
\end{proof}

\subsection{General type case} 
 
 \bigskip 
 Let $\dot G$ be the simply laced group with automorphism
 corresponding to $G$, as introduced in Section~\ref{s:automorphism}. 
 Note that we have already proved Proposition~\ref{p:main} for $\dot G$.
 Explicitly, we considered in Section~\ref{s:simplylaced} the projective 
 space determined by 
 the $\dot\rho$-representation of $\dot G$, that is, $\mathbb P(V(\dot \rho))$,
 with homogeneous coordinates determined by the canonical basis $\mathcal B(\dot \rho )$. 
 We have seen that for any 
 totally nonnegative cell $\dot {\mathcal R}^{>0}_{\dot v,\dot w}$ with parameterization 
 $\phi_{\dot{\mathbf v}_+,\dot{\mathbf w}}$, the composition $i\circ\phi_{\mathbf v_+,\mathbf w}$
 of this parameterization with the usual embedding,
 $i:\dot G/\dot B^+\to\mathbb P(V(\dot \rho))$, is given by polynomials
 with positive coefficients.
 
To treat the general case, we now identify $G$ with $\dot G^\tau$ and use
all of the notation from Section~\ref{s:automorphism}. For 
any $\bar i\in \bar I$
there is a simple reflection $s_{\bar i}$ in $W$, which is 
represented in $\dot G$ by
$$
s_{\bar i}:=\prod_{i\in\bar i}\dot s_i.
$$ 
In this way any reduced expression $\mathbf w=(\bar i_1,\bar i_2,\dotsc, \bar i_m)$ in $W$ gives rise to a reduced 
expression $\dot{\mathbf w}$ in $\dot W$ of length
$\sum_{k=1}^m |\bar i_k|$, which is determined
uniquely up to commuting elements \cite[Prop. 3.3]{Nanba}.
To a subexpression $\mathbf v$ of $\mathbf w$
we can then associate a unique subexpression 
$\dot {\mathbf v}$ of $\dot{\mathbf w}$ in 
the obvious way. 

\begin{lem}\label{l:folding} Let $\mathbf w$ be a reduced expression for $w$
in $W$ and $v\le w$.
If $\mathbf v_+$ is the positive subexpression for $v$
in $\mathbf w$, then $\mathbf{\dot{v}_+}$ is the positive subexpression for $v$ (now viewed as element of $\dot W$) in $\dot{\mathbf w}$.
\end{lem}

\begin{proof}
Let $\mathbf w = 
(\bar i_1,\bar i_2,\dotsc, \bar i_m)$ be a reduced expression 
for $w$ in $W$  and let 
$v_{(j)}$ denote the product (in order) of all simple generators
of $\mathbf v_+$ which come from $(\bar i_1, \bar i_2,\dotsc,\bar i_j)$.
The fact that $\mathbf v_+$ is positive in $W$ means that 
$v_{(j-1)} < v_{(j-1)} s_{\bar i_j}$ for all $j=1, \dots,n$.
In particular, 
$\ell (v_{(j-1)} s_{\bar i_j}) = \ell(v_{(j-1)})+1$ in $W$.
For the remainder of this proof let $\dot{v}_{(j-1)}$ denote the element
 $v_{(j-1)}$ viewed as an element of $\dot W$.
 Then by the relationship of lengths in $W$ versus $\dot W$,
 $\ell(\dot{v}_{(j-1)} \prod_{i_j \in \bar{i_j}} \dot{s}_{i_j})
   = \ell(\dot{v}_{(j-1)}) + |\bar{i_j}|$ in $\dot W$.
This fact together with the fact that the $\dot{s}_{i_j}$'s 
for $i_j \in \bar{i_j}$ commute with each other implies
that for any $i_j \in \bar{i_j}$, 
$\ell(\dot{v}_{(j-1)} \dot{s}_{i_j}) = \ell(\dot{v}_{(j-1)})+1$
in $\dot W$.
Therefore for any $i_j \in \bar{i_j}$,
$\dot{v}_{(j-1)} \dot{s}_{i_j} > \dot{v}_{(j-1)}$.
Letting $\bar{i_j}=\{i_{j_1},\dots,i_{j_r}\}$, this now shows that
$\dot{v}_{(j-1)} \dot{s}_{i_{j_1}} \dots \dot{s}_{i_{j_{k+1}}} > 
\dot{v}_{(j-1)} \dot{s}_{i_{j_1}} \dots \dot{s}_{i_{j_k}}$ for 
all $0 \leq k \leq r-1$ (note that we are again using 
the commutativity of the simple generators coming from $\bar{i_j}$).  A
little thought now
shows that
$\mathbf{\dot{v}_+}$ is positive in $\mathbf{ \dot{w}}$. 
\end{proof}

\begin{lem}\label{l:folding2} Let $v,w$ be in $W$ with $v\le w$.   
\begin{enumerate}
\item
We have 
$$
\mathcal R_{v,w}^{>0}=\dot{\mathcal R}_{v,w}^{>0}\cap \mathcal B^{\tau}.
$$
In particular the composition $i':\mathcal R_{v,w}\hookrightarrow \dot{\mathcal R}_{v,w}
\to \mathbb P(V(\dot\rho))$ is positivity preserving.
\item
Suppose $\mathbf w=(\bar i_1,\dotsc, \bar i_m)$ is a reduced expression for $w$ in $W$, and $\mathbf v_+=(j_1,\dotsc, j_k)$ is the positive subexpression for $v$. Then we have a commutative diagram,
\begin{equation*}
\begin{CD}
\mathcal R_{v,w}^{>0}& @>\iota>> &\dot{\mathcal R}_{v,w}^{>0}\\
@A\phi^{>0}_{\mathbf v_+,\mathbf w}AA & & @AA\phi^{>0}_{\dot{\mathbf v}_+,\dot{\mathbf w}}A\\
\R_{>0}^{\mathbf v_+^c} &@>\bar \iota>> & \R_{>0}^{\dot {\mathbf v}_+^c},
\end{CD}
\end{equation*}
where the top arrow is the usual inclusion, the vertical arrows are both isomorphisms,
and the map $\iota$  has the form 
$$(t_1,\dotsc, t_k)\mapsto (t_1,\dotsc, t_1, t_2,\dotsc, t_2,\dotsc, t_k),$$
 where each $t_l$ is repeated $|\bar i_{j_l}|$ times on the right hand side. 
\end{enumerate}
\end{lem}

\begin{proof}
(1) We have $\mathcal B_{\ge 0}=\dot{\mathcal B}_{\ge 0}\cap \dot{\mathcal B}^{\tau}$ by \cite{Lus:TotPos94}. 
Clearly $\mathcal R_{v,w}^{>0}\subset \dot{\mathcal R}_{v,w}^{>0}$. However since
$\mathcal B_{\ge 0}= \bigsqcup_{v,w\in W} \mathcal R_{v,w}^{>0}=\dot{\mathcal B}_{\ge 0}\cap \dot{\mathcal B}^{\tau}$ it follows that 
$\mathcal R_{v,w}^{>0}= \dot{\mathcal R}_{v,w}^{>0}\cap\dot{\mathcal B}^{\tau}$.

(2) This is a consequence of Lemma~\ref{l:folding}.
\end{proof}

We can now combine the two parts of Lemma~\ref{l:folding2} with 
Proposition~\ref{p:mainsimplylaced} 
to prove Proposition~\ref{p:main} for general type $G$.

\begin{proof}[Proof of Proposition~\ref{p:main}]
Firstly, Proposition~\ref{p:mainsimplylaced} gives the map 
$$
i\circ\phi^{>0}_{\dot{\mathbf v}_+,\dot{\mathbf w}}: \R_{>0}^{\dot {\mathbf v}_+^c}\to
\dot{\mathcal R}^{>0}_{v,w}\overset i\hookrightarrow\mathbb P(V(\dot\rho)).
 $$ 
Secondly, for non simply laced $G$ we will use 
the inclusion to projective space 
given by
$i':{\mathcal R}_{v,w}\hookrightarrow\mathbb P(V(\dot\rho))$, which is
positivity preserving by Lemma~\ref{l:folding2} (1). 
And thirdly, by Lemma~\ref{l:folding2} (2), we have that
$$
i'\circ\phi^{>0}_{{\mathbf v}_+,{\mathbf w}}: \R_{>0}^{ {\mathbf v}_+^c}\to
{\mathcal R}^{>0}_{v,w}\overset {i'}\hookrightarrow\mathbb P(V(\dot\rho))
$$
can be rewritten as 
$$i'\circ\phi^{>0}_{{\mathbf v}_+,{\mathbf w}}= i\circ\iota\circ\phi^{>0}_{{\mathbf v}_+,{\mathbf w}}
=(i\circ\phi^{>0}_{\dot{\mathbf v}_+,\dot{\mathbf w}})\circ\bar\iota.
$$
By Proposition~\ref{p:mainsimplylaced} for $\dot G$ 
and the description of $\bar\iota$ in Lemma~\ref{l:folding2} (2)
we see  that $i'\circ\phi^{>0}_{{\mathbf v}_+,{\mathbf w}}$ is given by positive 
polynomials. This proves Proposition~\ref{p:main} for $G$.
\end{proof}



\bigskip

\section{Generalization to partial flag varieties}\label{s:Partial}

In this section we generalize the previous results
to partial flag varieties.

\subsection{Lusztig's decomposition of $\mathcal P^J$}\label{s:strata}

The stratification of $\mathcal B$ into smooth pieces $\mathcal R_{v,w}$ has
an analogue for partial flag varieties introduced by Lusztig in 
\cite{Lus:PartFlag}. 

Consider a triple of Weyl group elements $x,u,w\in W$ with $x\in W^J_{max}$,
$w\in W^J$ and $u\in W_J$. Then 
$\mathcal P^J_{x,u,w}\subset \mathcal P^J$ is defined as the 
set of all $P\in \mathcal P^J$ such that there exist Borel subgroups 
$B_L$ and $B_R$ inside $P$ satisfying
\begin{equation*}
B^+\overset w\To B_L\overset u\To B_R\overset {x\inv w_0}\To B^-.
\end{equation*}
An equivalent characterization of $\mathcal P^J_{x,u,w}$ is 
\begin{equation*}
\mathcal P^J_{x,u,w}=\pi^J(\mathcal R_{x,w u})=\pi^J(\mathcal R_{x u\inv, w}).
\end{equation*}
It is not hard to see that $B_L$ and $B_R$ are uniquely determined
as the Borel subgroups in $P$ `closest to' $B^+$ respectively $B^-$
with regard to their relative position, and the projection maps 
$\mathcal R_{x,wu}\to \mathcal P^J_{x,u,v}$ and $\mathcal R_{x u\inv,w}\to
\mathcal P^J_{x,u,v}$ are isomorphisms. 
In particular $\mathcal P^J_{x,u,w}$ is nonempty if and only
if $x\le w u$, in which case it is smooth of dimension 
$\ell(w)+\ell(u)-\ell(x)$.  

Let us denote the indexing set for this decomposition of $\mathcal P^J$
by $Q^J$. So
\begin{equation*}
Q^J:=\{(x,u,w)\in W^J_{max}\x W_J\x W^J\ |\ x\le wu \}.
\end{equation*} 


\subsection{Totally nonnegative cells in 
$\mathcal P^J$}\label{s:P}

The totally positive and nonnegative parts of 
$\mathcal P^J$ are defined in
\cite{Lus:IntroTotPos} by
\begin{eqnarray*}
\mathcal P^J_{>0}&=&\pi^J(\mathcal B_{>0}),\\
\mathcal P^J_{\ge 0}&=&\pi^J(\mathcal B_{\ge 0}).
\end{eqnarray*} 
Since $\pi^J$ is closed it follows that $\mathcal P^J_{\ge 0}=
\overline{\mathcal P^J_{>0}}$.

We decompose $\mathcal P^J_{\ge 0}$ by intersecting it with the 
strata $\mathcal P^J_{x,u,w}$ from Section~\ref{s:strata}. 
From the definitions and the fact that reduction preserves total
positivity it follows that (\cite[Lemma 3.2]{Rie:PhD}) 
\begin{equation*}
\mathcal P^J_{x,u,w;>0}:=\mathcal P^J_{x,u,w}\cap\mathcal P^J_{\ge 0}
=\pi^J(\mathcal R_{x,wu}^{>0})=\pi^J(\mathcal R_{x u\inv, w}^{>0}).
\end{equation*}
Keeping in mind that 
$\pi^J:\mathcal R_{x,wu}\to\mathcal P^J_{x,u,w}$, say, 
is an isomorphism, we see that 
\begin{equation*}
\mathcal P^{J}_{x,u,w;>0}\cong \mathcal R^{>0}_{x,wu}\cong
\mathcal \R_{>0}^{\ell(w)+\ell(u)-\ell(x)},
\end{equation*}
for any triple $(x,u,w)\in Q^J$.

\begin{thm}\label{t:P}\cite{Rie:TotPosPoset} Let $(x,u,w)\in Q^J$.
Then
$\overline{\mathcal P^J_{x,u,w;>0}}$ is the disjoint union of
${\mathcal P^J_{x,u,w;>0}}$ and some lower-dimensional cells.

\end{thm}

\subsection{Proof of Theorem \ref{th:main}}

Finally we are ready to prove Theorem \ref{th:main}.
Let us recall the notion of a CW complex.

In a CW complex $X$, a 
cell is {\it attached} by
glueing a closed $i$-dimensional ball $D^i$ to the $(i-1)$-skeleton
$X_{i-1}$, i.e.\ the union of all lower dimensional cells.
The glueing is specified by a continuous function $f$ from
$\partial D^i = S^{i-1}$ to $X_{i-1}$.  CW complexes are defined
inductively as follows.  Given $X_0$ a discrete space
(a discrete union of $0$-cells), we inductively construct 
$X_i$ from $X_{i-1}$ by attaching some collection of $i$-cells.  The
resulting colimit space $X$ is called a {\it CW complex} provided
it is given the weak topology and the {\it closure-finite} condition
is satisfied for its closed cells.  Recall that the closure-finite
condition requires that every closed cell is covered by a finite
union of open cells.

\begin{proof}
The cell decomposition of 
$\mathcal P^J_{\ge 0}$ has 
only finitely many cells; therefore
the closure-finite condition in the definition of a CW complex
is automatically satisfied.
What we need to do is define the attaching maps for the cells.

Recall from Corollary \ref{c:glue} that for each parameterization of 
a cell 
${\mathcal R_{v,w}^{>0}}$ of $\mathcal B_{\geq 0}$ 
we have a toric variety 
$X_{\mathbf v_+,\mathbf w}$ and 
a map $\tau_{\geq 0}: X_{\mathbf v_+,\mathbf w}^{\ge 0}\to  \overline{\mathcal R_{v,w}^{>0}}$
which is surjective, 
an isomorphism on the strictly positive parts, 
and which takes the boundary of $X_{\mathbf v,\mathbf w}^{\ge 0}$ to the boundary of  $\overline{\mathcal R_{v,w}^{>0}}$.   Since  
$X_{\mathbf v_+,\mathbf w}^{\ge 0}$
is homeomorphic to a closed ball, this provides a glueing map for the cell
${\mathcal R_{v,w}^{>0}}$ 
of $\mathcal B_{\geq 0}$.

To construct the glueing map for 
the cell $P^J_{x,u,w;>0}$ of 
$\mathcal P^J_{\ge 0}$,  we compose the map $\pi^J$ from 
$\mathcal B_{\ge 0}$ to 
$\mathcal P^J_{\ge 0}$ with $\tau_{\geq 0}$.  
Since 
$\pi^J:\mathcal R_{x,wu}^{>0}\to\mathcal P^J_{x,u,w;>0}$ is a homeomorphism,
we obtain a map 
$\pi^J \circ \tau_{\geq 0} $ from 
$X_{\mathbf x_+,\mathbf {wu}}^{\ge 0}\to  \overline{\mathcal P^J_{x,u,w;>0}}$
which again is surjective and  
an isomorphism on the strictly positive parts.   

Now as in the proof of Corollary \ref{c:glue}, it must take
the boundary of 
$X_{\mathbf x_+,\mathbf wu}^{\ge 0}$ to the boundary of 
$\overline{\mathcal P^J_{x,u,w;>0}}$.  
Therefore 
the composition $\tau_{\geq 0} \circ \pi^J$ provides a glueing
map for cells of 
$\mathcal P^J_{\ge 0}$. 

The only thing that remains to check is that the boundary of each
$i$-cell is mapped to the $i-1$-skeleton.  But this follows from
Theorem \ref{t:P}.
\end{proof}

As explained in the Introduction, Theorem \ref{th:main} together with
Theorem \ref{th:Eulerian} from \cite{Williams:poset} implies the following
result.

\begin{cor}
The Euler characteristic of the closure of a cell of 
$(G/P)_{\ge 0}$ is $1$.  
\end{cor}

\bibliographystyle{amsplain}

\def\cprime{$'$}
\providecommand{\bysame}{\leavevmode\hbox to3em{\hrulefill}\thinspace}
\providecommand{\MR}{\relax\ifhmode\unskip\space\fi MR }
\providecommand{\MRhref}[2]{%
  \href{http://www.ams.org/mathscinet-getitem?mr=#1}{#2}
}

\providecommand{\href}[2]{#2}

\end{document}